\documentclass[a4paper]{amsart}
\usepackage{amssymb}
\usepackage{amsmath}
\usepackage{amscd}
\usepackage[all]{xy}

\newtheorem{theorem}{Theorem}
\newtheorem{lemma}{Lemma}
\newtheorem{proposition}{Proposition}
\newtheorem{corollary}{Corollary}

\theoremstyle{definition}
\newtheorem{definition}{Definition}
\newtheorem{remark}{Remark}
\newtheorem{example}{Example}
\newtheorem{question}{Question}
\newtheorem{notation}{Notation}

\newcommand{\p}{{\mathbb P}}

\newcommand{\q}{{\mathbb Q}}
\renewcommand{\c}{{\mathbb C}}

\newcommand{\g}{{\mathbb G}}
\newcommand{\z}{{\mathbb Z}}

\renewcommand{\O}{{\mathcal O}}

\newcommand{\da}{{\dashrightarrow}}

\newcommand{\pic}{\operatorname{Pic}}
\newcommand{\length}{\operatorname{length}}

\newcommand{\aut}{\operatorname{Aut}}
\newcommand{\mor}{\operatorname{Mor}}
\begin{document}

\title[Quadro-quadric special birational transformations of $\p^r$]{Quadro-quadric special birational transformations of projective spaces}
\author[Alberto Alzati \and Jos\'{e} Carlos Sierra]{Alberto Alzati* \and Jos\'{e} Carlos Sierra**}

\address{Dipartimento di Matematica, Universit\`a degli Studi di Milano\\
via Cesare Saldini 50, 20133 Milano, Italy}
\email{alberto.alzati@unimi.it}

\address{Instituto de Ciencias Matem\'aticas (ICMAT), Consejo Superior de Investigaciones Cient\'{\i}ficas (CSIC), Campus de
Cantoblanco, 28049 Madrid, Spain}
\email{jcsierra@icmat.es}

\thanks{* This work is within the framework of the national research project
``Geometria delle Variet\`a Algebriche" PRIN 2010 of MIUR}

\thanks {** Research supported by the ``Ram\'on y Cajal" contract
RYC-2009-04999, the project MTM2009-06964 of MICINN and the ICMAT ``Severo Ochoa" project SEV-2011-0087 of MINECO}
\date{\today}
\begin{abstract}
Special birational transformations $\Phi:\p^r\da Z$ defined by quadric hypersurfaces are studied by means of the variety of lines $\mathcal L_z\subset\p^{r-1}$ passing through a general point $z\in Z$. Classification results are obtained when $Z$ is either a Grassmannian of lines, or the $10$-dimensional spinor variety, or the $E_6$-variety. In the particular case of quadro-quadric transformations, we extend the well-known classification of Ein and Shepherd-Barron coming from Zak's classification of Severi varieties to a wider class of prime Fano manifolds $Z$. Combining both results, we get a classification of special birational transformations $\Phi:\p^r\da Z$ defined by quadric hypersurfaces onto (a linear setion of) a rational homogeneous variety different from a projective space and a quadric hypersurface.
\end{abstract}
\maketitle

\section{Introduction}
We deal with birational transformations $\Phi:\p^r\da Z$ from the complex projective space onto a prime Fano manifold whose base locus $X\subset\p^r$ is a smooth irreducible and reduced scheme of dimension $n$. According to the classical terminology of Cremona transformations, these birational transformations are called \emph{special} and they have been recently studied in \cite{alzati-sierra:cremona}, extending some of the main results of \cite{c-k}, \cite{e-sb} and \cite{c-k2}. In particular, complete classifications have been obtained in the cases $n=1$, $n=2$ and $r-n=2$.

In this paper, we consider special birational transformations defined by quadric hypersurfaces. In this case, the techniques involved are different from those present in \cite{alzati-sierra:cremona}. This is due to a couple of facts:

- The secant lines of $X\subset\p^r$ are contracted by $\Phi$, and hence the theory of secant defective manifolds plays a central role in the study of $X\subset\p^r$;

- $Z$ is covered by lines, which allows us to adopt the strategy explained at the end of this introduction.

Special Cremona transformations of type $(2,2)$ were classified in \cite[Theorem 2.6]{e-sb} thanks to Zak's classification of Severi varieties (see \cite{zak-severi} or \cite[Ch. IV]{zak2}; see also \cite{l-vdv}). However, in view of \cite[\S 4]{russo}, it looks rather difficult to obtain complete classification results on special Cremona transformations of type $(2,b)$, or even to construct new examples. This being said, assume now that $Z$ is a prime Fano manifold different from $\p^r$. A bunch of examples of special birational transformations of type $(2,1)$ is given by the orbit of the highest weight vector of an irreducible representation of some algebraic groups, and is described in detail in \cite[Ch. III, Theorem 3.8]{zak2}. More precisely, in these examples either $X\subset\p^{2n}$ is a degenerate Segre embedding of $\p^1\times\p^{n-1}\subset\p^{2n-1}$ and $Z\subset\p^{(n^2+3n)/2}$ is the Pl\"ucker embedding of the Grassmannian of lines $\g(1,n+1)$, or $X\subset\p^{10}$ is a degenerate embedding of $\g(1,4)\subset\p^9$ and $Z\subset\p^{15}$ is the minimal embedding of the $10$-dimensional spinor variety $S_4$, or $X\subset\p^{16}$ is a degenerate embedding of $S_4\subset\p^{15}$ and $Z\subset\p^{26}$ is the minimal embedding of the $E_6$-variety. A series of examples of special birational transformations of type $(2,2)$, that we call \emph{quadro-quadric} in view of the classical terminology, was given by Semple in \cite{semple}: the image of the quadric hypersurfaces containing a rational normal scroll $X\subset\p^{2n+2}$ of degree $n+3$ is a Grassmannian of lines $\g(1,n+2)\subset\p^{(n^2+5n+4)/2}$. Moreover, if we restrict a special quadro-quadric Cremona transformation $\Phi:\p^r\da\p^r$ to a general hyperplane $H\subset\p^r$ we get a special birational transformation $\Phi_{|H}:\p^{r-1}\da Z\subset\p^r$ onto a quadric hypersurface whose base locus $X\subset\p^{r-1}$ is a hyperplane section of a Severi variety. On the contrary, very few examples of special birational transformations of type $(2,b)$ are known for $b\geq 3$ (cf. \cite{alzati-sierra:cremona}, \cite{stagliano} and \cite{stagliano2}).

With this examples in mind, we prove the following results. On the one hand, we get a classification of special birational transformations $\Phi:\p^r\da Z$ defined by quadric hypersurfaces when $Z$ is either a Grassmannian of lines, or the $10$-dimensional spinor variety, or the $E_6$-variety (see Theorems \ref{thm:grassmannians}, \ref{thm:spinor} and \ref{thm:cartan}). In particular, this extends some of the main results of \cite{r-s} (see \cite[Theorems 5.2 and 5.3]{r-s}). On the other hand, extending the classification of quadro-quadric special Cremona transformations, we get in Theorem \ref{thm:(2,2)} a classification of quadro-quadric special birational transformations $\Phi:\p^r\da Z$ when $Z$ belongs to a wide class of manifolds (namely LQEL-manifolds, see \cite[Definition 1.1]{russo}) that includes, among others, (linear sections of) rational homogeneous varieties. This is actually the core of the paper. Combining both results, we get in Corollary \ref{cor:homogeneous} a classification of special birational transformations $\Phi:\p^r\da Z$ defined by quadric hypersurfaces when $Z$ is (a linear section of) a rational homogeneous variety different from a projective space and a quadric hypersurface (cf. Remark \ref{rem:hom}).

The main strategy to prove these results comes from the study of uniruled varieties carried out after \cite{mori} and \cite{m-m}, in the following way. Let $z\in Z$ be a general point and let $z:=\Phi(p)$. Then $\Phi(\langle p,x \rangle)\subset Z$ is a line passing through $z$ for every $x\in X$, so $Z$ is covered by lines and we can consider the variety ${\mathcal L}_z\subset\p^{r-1}$ parameterizing lines passing through $z\in Z$. In this setting, we can see ${\mathcal L}_z\subset\p^{r-1}$ as the \emph{variety of minimal rational tangents} introduced by Hwang and Mok in \cite{h-m1} and developed in a series of papers (for a detailed account of their results and applications to concrete geometric problems see for instance the surveys \cite{h-m2}, \cite{hwang}, \cite{mok} and \cite{hwang1}). Hence we can apply their basic philosophy in our context: namely, that many problems on $Z$ can be solved by using the projective geometry of ${\mathcal L}_z\subset\p^{r-1}$. This approach, which has a more projective flavor, has been recently used by Ionescu and Russo in \cite{russo}, \cite{i-r4} and \cite{i-r5}. Finally, the description of $X\subset\p^r$ as a hyperplane section of a Severi variety in the most difficult part of Theorem \ref{thm:(2,2)} (see Proposition \ref{prop:main}) is achieved by putting together Zak's ideas on the classification of Severi varieties and Russo's ideas on the study of secant defective LQEL-manifolds by means of $\mathcal L_x\subset\p^{n-1}$ (see \cite{russo}).

\section{Preliminaries and first results}

Let $f_0,\dots,f_{\alpha}\in \c[X_0,\dots,X_r]$ be homogeneous
polynomials of degree $a\geq 2$. Let $\Phi:\p^r\da\p^{\alpha}$ be
the corresponding rational map, and let $Z:=\Phi(\p^r)$. We assume
that $Z$ is smooth and that $\Phi:\p^r\da Z$ is a birational map.
Let $\Psi:Z\da\p^r$ be the inverse of $\Phi:\p^r\da Z$. Let $X\subset\p^r$ and $Y\subset Z$
denote the the base (also called fundamental) locus of $\Phi$ and $\Psi$, respectively. We
assume that $X$ is a smooth irreducible and reduced
scheme. We point out that $X\subset\p^r$ is allowed to be degenerate, i.e. contained in a hyperplane of $\p^r$. Let $n:=\dim(X)$ and
$m:=\dim(Y)$.

In \cite{alzati-sierra:cremona}, we studied the case in which $Z$ is a manifold with cyclic Picard group generated by the hyperplane section $\O_Z(1)$ of $Z\subset\p^{\alpha}$. We obtained a complete classification for $n\in\{1,2\}$, $m=\{0,1,2\}$ and $r=n+2$, as well as some partial results for $n=3$. In this paper we focus on the case $a=2$. Let $W$ denote the blowing-up of $\p^r$ along $X$,
with projection maps $\sigma:W\to\p^r$ and $\tau:W\to Z$ such that
$\tau=\Phi\circ\sigma$. A line $L\subset\p^r$ is a secant line of $X$ if $L\not\subset
X$ and $\length(X\cap L)\geq 2$. Let $SX\subset\p^r$ denote
the closure of the union of all the secant lines of $X$. In this case, the secant lines of $X\subset\p^r$ are contracted by $\Phi$, so we get the following:

\begin{proposition}\label{prop:linear}
Let $\Phi:\p^r\da Z$ be a special birational transformation defined by quadric hypersurfaces. Then $\rho(Z)\leq 2$ and equality holds if and only if $X\subset\p^r$ is a linear subspace.
\end{proposition}

\begin{proof}
Consider the morphism $\tau:W\to Z$. As $\rho(W)=2$, we deduce $\rho(Z)\leq 2$ and equality holds if and only if $\tau$ is an isomorphism. This happens if and only if there is no secant line of $X\subset\p^r$, that is, if and only if $X\subset\p^r$ is a linear subspace.
\end{proof}

Therefore we can assume $\rho(Z)=1$. Moreover, as $Z$ is covered by lines, the Picard group of $Z$ is generated by the hyperplane section $\O_Z(1)$ of $Z\subset\p^{\alpha}$. Let $H:=\sigma^*\O_{\p^r}(1)$ and
$H_Z:=\tau^*\O_Z(1)$. Let $E$ be the exceptional divisor of
$\sigma:W\to\p^r$, and let $E_Z:=\tau^{-1}(Y)$ (scheme
theoretically). Then $E_Z$ is irreducible and reduced (see \cite[Proposition 1]{alzati-sierra:cremona}), so we get $H_z=2H-E$ and $H=bH_Z-E_Z$ for some integer $b\geq 1$. In this setting, we say that $\Phi:\p^r\da Z$ is a \emph{special birational transformation of $\p^r$ of type $(2,b)$}. The useful remark given in \cite[Proposition 2.3]{e-sb} also holds in our setting:

\begin{proposition}\label{prop:sec}
Notation as above:
\begin{enumerate}
\item[(i)] $\sigma(E_Z)=SX\subset\p^r$ is a hypersurface of degree $2b-1$;
\item[(ii)] Let $u\in\sigma(E_Z)$ be a general point. Then the union of all the secant lines of $X\subset\p^r$ through $u$ is an $(r-m-1)$-plane that intersects $X$ in a quadric hypersurface.
\end{enumerate}
\end{proposition}

Let $\delta:=\delta(X)$ denote the secant defect of $X\subset\p^r$, that is, $\delta=2n+1-\dim(SX)$. For any $u\in SX$, let $\Sigma_u\subset X$ denote the \emph{entry locus of $u$}, that is, the closure of the set $\{x\in X\mid\exists x'\in X\, \textrm{with}\,\, u\in\langle x,x' \rangle\}$. We recall that $\dim(\Sigma_u)=\delta$ for general $u\in SX$. Then we immediately get from Proposition \ref{prop:sec}:

\begin{corollary}\label{cor:sec}
Notation as above:
\begin{enumerate}
\item[(i)] $r=2n+2-\delta$;
\item[(ii)] $\Sigma_u\subset\p_u^{\delta+1}$ is a quadric hypersurface for general $u\in SX$ and $m=2n-2\delta$.
\end{enumerate}
\end{corollary}

\begin{remark}\label{rem:entry}
Furthermore, we also deduce from the proof of \cite[Proposition 2.3]{e-sb} that $\Sigma_v=\Sigma_u$ for every $v\in\p_u^{\delta+1}-\Sigma_u$.
\end{remark}

\begin{remark}\label{rem:QEL}
Manifolds $X\subset\p^r$ for which $\Sigma_u\subset\p_u^{\delta+1}$ is a quadric hypersurface for general $u\in SX$ have been studied by several authors. According to the terminology introduced in \cite{russo}, we will say that $X\subset\p^r$ is a QEL-manifold.
\end{remark}

Let us compute the index $i:=i(Z)$ of the prime Fano manifold $Z$:

\begin{proposition}\label{prop:num}
Let $\Phi:\p^r\da Z$ be a special birational transformation of type
$(2,b)$. Then $i=n+2+(b-1)(\delta+1)$.
\end{proposition}

\begin{proof}
It follows from \cite[Proposition 3]{alzati-sierra:cremona} and Corollary \ref{cor:sec}(i).
\end{proof}

We can deduce several numerical constraints from Proposition \ref{prop:num} and Corollary \ref{cor:sec}(i). For instance, we easily get the following results:

\begin{proposition}\label{prop:grassmannians}
Let $\Phi:\p^r\da Z$ be a special birational transformation of type $(2,b)$. If $Z$ is (a codimension-$s$ linear section of) the Grassmannian of lines $\g(1,q+1)$ then one of the following holds:
\begin{enumerate}
\item[(i)] $b=1$, $\delta=2-s$, $s\in\{0,1,2\}$, $n=q-s$;
\item[(ii)] $b=2$, $\delta=0$, $s=0$, $n=q-1$.
\end{enumerate}
\end{proposition}

\begin{proof}
In this case $r=2q-s=2n+2-\delta$ and $i=q+2-s=n+2+(b-1)(\delta+1)$, so the statement follows from a simple computation.
\end{proof}

\begin{proposition}\label{prop:spinor}
Let $\Phi:\p^r\da Z$ be a special birational transformation of type $(2,b)$. If $Z$ is (a codimension-$s$ linear section of) the $10$-dimensional spinor variety $S_4$ then one of the following holds:
\begin{enumerate}
\item[(i)] $b=1$, $\delta=4-s$, $s\in\{0,1,2,3,4\}$, $n=6-s$;
\item[(ii)] $b=2$, $\delta=0$, $s=2$, $n=3$;
\item[(iii)] $b=3$, $\delta=0$, $s=0$, $n=4$.
\end{enumerate}
\end{proposition}

\begin{proof}
In this case $r=10-s=2n+2-\delta$ and $i=8-s=n+2+(b-1)(\delta+1)$, so the statement follows from a simple computation.
\end{proof}

\begin{proposition}\label{prop:cartan}
Let $\Phi:\p^r\da Z$ be a special birational transformation of type $(2,b)$. If $Z$ is (a codimension-$s$ linear section of) the $E_6$-variety then one of the following holds:
\begin{enumerate}
\item[(i)] $b=1$, $\delta=6-s$, $s\in\{0,1,2,3,4,5,6\}$, $n=10-s$;
\item[(ii)] $b=2$ and either $\delta=0$, $s=4$, $n=5$, or else $\delta=1$, $s=1$, $n=7$;
\item[(iii)] $b=3$, $\delta=0$, $s=2$, $n=6$;
\item[(iv)] $b=4$, $\delta=0$, $s=0$, $n=7$.
\end{enumerate}
\end{proposition}

\begin{proof}
In this case $r=16-s=2n+2-\delta$ and $i=12-s=n+2+(b-1)(\delta+1)$, so the statement follows from a simple computation.
\end{proof}

\begin{remark}
By \emph{codimension-$s$ linear section} we mean a codimension-$s$ subvariety of $\g(1,q+1)$ (resp. $S_4$ and $E_6$) obtained by intersecting $\g(1,q+1)$ (resp. $S_4$ and $E_6$) with a suitable codimension-$s$ linear subspace in the ambient projective space of $\g(1,q+1)$ (resp. $S_4$ and $E_6)$.
\end{remark}

\begin{remark}
The fact of considering special birational transformations onto (linear sections of) Grassmannians of lines, the $10$-dimensional spinor variety and the $E_6$-variety might seem rather arbitrary at the moment. The reader will understand the motivation of considering these (linear sections of) secant defective rational homogeneous varieties in what follows.
\end{remark}

\subsection{The variety of lines passing through a general point of $Z$}

Let $z\in Z$ be a general point and let $z:=\Phi(p)$. As $p\in\p^r-SX$ then $\langle p,x \rangle\cap X=x$ for every $x\in X$, so we deduce that $\Phi(\langle p,x \rangle)\subset Z$ is a line passing through $z$ for every $x\in X$ and that $Z\subset\p^{\alpha}$ is covered by lines. Let $\mathcal L$ be an irreducible covering family of lines in $Z$ of maximal dimension, and ${\mathcal L}_z\subset\mathcal L$ denote the variety of lines passing through a general point $z\in Z$. In particular, the above argument gives an inclusion $X\subset\mathcal L_z$. In this setting, we can see ${\mathcal L}$ as a \emph{minimal rational component of the Chow space of $Z$} and $\mathcal L_z\subset\p^{r-1}$ as the \emph{variety of minimal rational tangents} studied by Hwang and Mok in a series of papers. We refer to \cite{h-m2}, \cite{hwang}, \cite{mok} and \cite{hwang1} for a nice account. Their basic philosophy applied in our context says that \emph{many problems on $Z$ can be solved by using the projective geometry of ${\mathcal L}_z\subset\p^{r-1}$}, as ${\mathcal L}_z$ is expected to be \emph{simpler} than $Z$. In the sequel, we will get some results (namely Theorems \ref{thm:grassmannians}, \ref{thm:spinor}, \ref{thm:cartan} and \ref{thm:(2,2)}, as well as Corollary \ref{cor:homogeneous}) by applying this principle in our setting. The following result has several consequences:

\begin{proposition}\label{prop:X_p}
If $\Phi:\p^r\da Z$ is a special birational transformation defined by quadric hypersurfaces then $X\subset\p^r$ is embedded in ${\mathcal L}_z\subset\p^{r-1}$ by the linear projection $\pi_p:X\to\p^{r-1}$ from $p\in\p^r - SX$.
\end{proposition}

\begin{proof}
We can assume $p=(1:0:\dots:0)$, $z=(1:0:\dots:0)$ and $T_zZ\subset\p^{\alpha}$ defined by $\{Z_{r+1}=\dots=Z_{\alpha}=0\}$, where $T_zZ$ is the embedded tangent projective space to $Z\subset\p^{\alpha}$ at $z$ and $(Z_0:\dots:Z_{\alpha})$ are the coordinates of $\p^{\alpha}$. Therefore we can assume that $f_i=X_0X_i+g_i$ with $g_i\in\c[X_1,\dots,X_r]$ for $i\in\{0,\dots,r\}$ and that $f_j\in\c[X_1,\dots,X_r]$ for $j\in\{r+1,\dots,\alpha\}$, where $(X_0:\dots:X_r)$ are the coordinates of $\p^r$. Note that the line $\Phi(\langle p,x \rangle)\subset Z$ is the line corresponding to the coordinates $(X_1:\dots:X_r)$ of the $\p^{r-1}$ giving the lines in $T_zZ$ passing through $z$. Hence $X\subset\p^r$ is embedded in ${\mathcal L}_z\subset\p^{r-1}$ by the linear projection $\pi_p:X\to\p^{r-1}$ from $p$.
\end{proof}

\begin{notation}\label{not:X_p}
According to Proposition \ref{prop:X_p}, we will denote by $X_p$ the projection of $X$ by $\pi_p$. So we get $X_p\subset{\mathcal L}_z\subset\p^{r-1}$.
\end{notation}

\begin{remark}\label{rem:tangential}
Let $t_z:Z\da t_z(Z)\subset\p^{\alpha-r-1}$ denote the tangential projection of $Z\subset\p^{\alpha}$ from $T_zZ$. Then the composition $t_z\circ\Phi:\p^r\da t_z(Z)$ is given by the quadratic polynomials $f_j\in\c[X_1,\dots,X_r]$, where $j\in\{r+1,\dots,\alpha\}$, so it induces a map $\p^{r-1}\da t_z(Z)$. Note that $\mathcal L_z\subset\p^{r-1}$ is  contained in the base scheme of $f_{r+1},\dots,f_{\alpha}$, and recall that $\dim(t_z(Z))=\dim(Z)-\delta(Z)$.
\end{remark}

\begin{corollary}\label{cor:L_z}
If $\Phi:\p^r\da Z$ is a special birational transformation of type $(2,b)$ and $z\in Z$ is a general point then:
\begin{enumerate}
\item[(i)] ${\mathcal L}_z\subset\p^{r-1}$ is smooth and non-degenerate;
\item[(ii)] $\dim({\mathcal L}_z)=n+(b-1)(\delta+1)$;
\item[(iii)] $S{\mathcal L}_z=\p^{r-1}$ and $\delta({\mathcal L}_z)=2b-2+(2b-1)\delta$;
\item[(iv)] if $b=1$ then $X\subset\p^r$ is a degenerate embedding and ${\mathcal L}_z\subset\p^{r-1}$ is projectively equivalent to $X\subset\p^{r-1}$;
\item[(v)] $\mathcal L_z$ is irreducible;
\item[(vi)] if $b\geq 2$ then $\mathcal L_z$ has cyclic Picard group unless $b=2$ and $\delta=0$;
\item[(vii)] if $b\geq 2$ then $Z$ is covered by planes; furthermore, if $b=2$ then ${\mathcal L}_z\subset\p^{r-1}$ is covered by lines, so in particular it is a prime Fano manifold unless $\delta=0$.
\end{enumerate}
\end{corollary}

\begin{proof} (i) The smoothness of ${\mathcal L}_z$ follows, for instance, from \cite[Proposition 1.5]{hwang}), and the non-degeneracy of ${\mathcal L}_z\subset\p^{r-1}$ follows from Propositions \ref{prop:sec}(i) and \ref{prop:X_p} (cf. \cite[Theorem 2.5]{hwang}).

(ii) Let $l\in\mathcal L$, and let $L\subset Z$ denote the corresponding line. A standard deformation argument shows that $i(Z)=-K_Z\cdot L=2+\dim({\mathcal L}_z)$, and hence $\dim({\mathcal L}_z)=n+(b-1)(\delta+1)$ by Proposition \ref{prop:num}.

(iii) As $SX\subset\p^r$ is a hypersurface by Proposition \ref{prop:sec}(i), we get $SX_p=\p^{r-1}$ and hence $S{\mathcal L}_z=\p^{r-1}$. Since $r=2n+2-\delta$ by Corollary \ref{cor:sec}(i) and $S{\mathcal L}_z=\p^{r-1}$, a simple computation shows that $\delta({\mathcal L}_z)=2b-2+(2b-1)\delta$.

(iv) If $b=1$ then $SX=\p^{r-1}$ by Proposition \ref{prop:sec}(i), so $X\subset\p^r$ is a degenerate embedding and ${\mathcal L}_z\subset\p^{r-1}$ is projectively equivalent to $X\subset\p^{r-1}$ by Proposition \ref{prop:X_p}.

(v) If $b=1$ then $\mathcal L_z$ is irreducible by (iv), so we can assume $b\geq 2$. We deduce from \cite[Theorem 1.3]{hwang} that $\mathcal L_z$ is equidimensional. As $\mathcal L_z$ is smooth and $2\dim(\mathcal L_z)\geq r-1$ we get that $\mathcal L_z$ is irreducible.

(vi) If $b\geq 2$ then $\delta(\mathcal L_z)=2b-2+(2b-1)\delta\geq3$ unless $b=2$ and $\delta=0$, so we conclude by the Barth and Larsen theorems.

(vii) Let $S^2X\subset\p^r$ denote the closure of the union of the $3$-secant planes of $X\subset\p^r$. As $SX\subset\p^r$ is a non-linear hypersurface by Proposition \ref{prop:sec}(i), we deduce $SX\subsetneq S^2X=\p^r$. Therefore, for every $p\in\p^r - SX$ there passes a $3$-secant plane of $X$ which is mapped by $\Phi$ onto a plane contained in $Z$, and hence $Z$ is covered by planes. If $b=2$, as a general line of $Z$ passing through $z$ corresponds to a $3$-secant conic of $X$ in $\p^r$ passing through $p$ (and hence it is contained in a plane of $Z$ through $z$), we deduce that ${\mathcal L}_z\subset\p^{r-1}$ is covered by lines. Combined with (vi), this gives the last assertion.
\end{proof}

\begin{remark}
Corollary \ref{cor:L_z}, (vi) and (vii), is sharp. If $\Phi:\p^{2n+2}\da\g(1,n+2)$ is the special birational transformation of type $(2,2)$ given by the quadric hypersurfaces defining a rational normal scroll of degree $n+3$ (see \cite{semple}) then $\mathcal L_z=\p^1\times\p^n$. On the other hand, if $\Phi:\p^4\da Z$ is the special birational transformation of type $(2,1)$ given by the quadric hypersurfaces defining a twisted cubic (see \cite[Theorem 4(III)]{alzati-sierra:cremona}) then $Z$ is a codimension-$2$ linear section of $\g(1,4)$ which is not covered by planes.
\end{remark}

\subsection{Some classification results}
According to the above section, we are able to classify some special quadratic birational transformations $\Phi:\p^r\da Z$ when $Z$ is a concrete prime Fano manifold (namely a Hermitian symmetric space) for which $\mathcal L_z\subset\p^{r-1}$ is well known (see for instance \cite[\S 1.4.5]{hwang}):

\begin{theorem}[{cf. \cite[Theorems 5.2 and 5.3]{r-s}}]\label{thm:grassmannians}
Let $\Phi:\p^r\da\g(1,r/2+1)$ be a special birational transformation of type $(2,b)$ onto a Grassmannian of lines. Then one of the following holds:
\begin{enumerate}
\item[(i)] $b=1$ and $X\subset\p^{2n}$ is a degenerate Segre embedding of $\p^1\times\p^{n-1}$;
\item[(ii)] $b=2$ and $X\subset\p^{2n+2}$ is a rational normal scroll of degree $n+3$.
\end{enumerate}
\end{theorem}

\begin{proof}
In view of Proposition \ref{prop:grassmannians}, we consider first the case $b=1$, $\delta=2$ and $r=2n$. In this case, $X\subset\p^{2n-1}$ is projectively equivalent to ${\mathcal L}_z\subset\p^{2n-1}$ by Corollary \ref{cor:L_z}(iv). Therefore $X\subset\p^{2n}$ is a degenerate Segre embedding of $\p^1\times\p^{n-1}\subset\p^{2n-1}$.

Assume now $b=2$, $\delta=0$ and $r=2n+2$. Let $X_p\subset {\mathcal L}_z=\p^1\times\p^n\subset\p^{2n+1}$ (cf. Notation \ref{not:X_p}), and let $p'\in\p^{2n+1}$ be a general point. The entry locus of $p'$ (relative to ${\mathcal L}_z$) is a $2$-dimensional quadric $\Sigma_{p'}\subset\p_{p'}^3$ intersecting $X_p$ in a curve $C_{p'}\subset\p_{p'}^3$. Let $\p_{pp'}^4:=\langle p,\p_{p'}^3\rangle$, and let $C_{pp'}:=\pi_p^{-1}(C_{p'})$. Then $\Phi:\p^{2n+2}\da\g(1,n+2)$ restricted to $\p_{pp'}^4\subset\p^{2n+2}$ is a birational transformation onto a smooth $4$-dimensional quadric hypersurface $\phi(\p_{pp'}^4)\subset\g(1,n+2)$, namely a fibre of the tangential projection of $\g(1,n+2)$ onto $\g(1,n)$ (see Remark \ref{rem:tangential}). Repeat the construction taking a general $p''\in\p^{2n+1}$ such that $\Sigma_{p'}\cap\Sigma_{p''}$ is a line (note that this is always possible). Then $\p_{pp'}^4\cap\p_{pp''}^4$ is a plane, say $\p^2$, and hence $\Phi(\p^2)\subset\g(1,n+2)$ is also a plane, namely the intersection of the $4$-dimensional quadric hypersurfaces $\Phi(\p_{pp'}^4)$ and $\Phi(\p_{pp''}^4)$ of $\g(1,n+2)$. As $\Sigma_{p'}\cap\Sigma_{p''}\cap X_p$ is a finite subscheme $\Lambda$ of length $\lambda$ and $\Phi_{|\p^2}:\p^2\da\Phi(\p^2)$ is a Cremona transformation of type $(2,2)$, we get $\lambda=3$. Since $\deg(SX)=3$, we deduce that $\deg(SC_{pp'})=3$ and that $C_{pp'}\subset\p_{pp'}^4$ is a rational normal quartic curve (cf. \cite[Theorem 4(II)]{alzati-sierra:cremona}). Therefore $C_{p'}\subset\Sigma_{p'}$ (and hence $X_p\subset\mathcal L_z=\p^1\times\p^n$) is a divisor of type $(1,3)$, so $X\subset\p^{2n+2}$ is a rational normal scroll of degree $n+3$.
\end{proof}

In a similar way, we can prove the following results:

\begin{theorem}\label{thm:spinor}
Let $\Phi:\p^{10}\da S_4$ be a special birational transformation of type $(2,b)$ onto the $10$-dimensional spinor variety. Then $b=1$ and $X\subset\p^{10}$ is a degenerate Pl\"ucker embedding of $\g(1,4)$.
\end{theorem}

\begin{proof}
In view of Proposition \ref{prop:spinor}, we consider first the case $b=1$, $\delta=4$ and $n=6$. In this case, $X\subset\p^9$ is projectively equivalent to ${\mathcal L}_z\subset\p^9$ by Corollary \ref{cor:L_z}(iv). Therefore $X\subset\p^{10}$ is a degenerate Pl\"ucker embedding of $\g(1,4)\subset\p^9$.

Assume now $b=3$, $\delta=0$ and $n=4$. Let $X_p\subset {\mathcal L}_z=\g(1,4)\subset\p^9$ (cf. Notation \ref{not:X_p}), and let $p'\in\p^9$ be a general point. The entry locus of $p'$ (relative to ${\mathcal L}_z$) is a $4$-dimensional quadric hypersurface $\Sigma_{p'}\subset\p_{p'}^5$ intersecting $X_p$ in a surface. Let $\p_{pp'}^6:=\langle p,\p_{p'}^5\rangle$. Then $\Phi:\p^{10}\da S_4$ restricted to $\p_{pp'}^6\subset\p^{10}$ is a birational transformation onto a smooth $6$-dimensional quadric hypersurface $\Phi(\p_{pp'}^6)\subset S_4$, namely a fibre of the tangential projection of $S_4$ onto $\p^4$ (see Remark \ref{rem:tangential}). Repeat the construction taking a general $p''\in\p^9$ such that $\Sigma_{p'}\cap\Sigma_{p''}$ is a plane (in this case, the condition on $p''$ is automatic). Then $\p_{pp'}^6\cap\p_{pp''}^6$ is a linear $\p^3$ and hence $\Phi(\p^3)\subset S_4$ is also a linear $\p^3$, namely the intersection of the $6$-dimensional quadric hypersurfaces $\Phi(\p_{pp'}^6)$ and $\Phi(\p_{pp''}^6)$ of $S_4$. As $\Sigma_{p'}\cap\Sigma_{p''}\cap X_p$ is a finite subscheme $\Lambda$, the Cremona transformation $\Phi_{|\p^3}:\p^3\da\Phi(\p^3)$ induces a small contraction from the blowing-up of $\p^3$ along $\pi_p^{-1}(\Lambda)$ onto $\Phi(\p^3)$, and we get a contradiction since $\Phi(\p^3)$ would be singular.
\end{proof}

\begin{theorem}\label{thm:cartan}
Let $\Phi:\p^{16}\da E_6$ be a special birational transformation of type $(2,b)$ onto the $E_6$-variety. Then $b=1$ and $X\subset\p^{16}$ is a degenerate minimal embedding of $S_4$.
\end{theorem}

\begin{proof}
In view of Proposition \ref{prop:cartan}, we consider first the case $b=1$, $\delta=6$ and $n=10$. In this case, $X\subset\p^{15}$ is projectively equivalent to ${\mathcal L}_z\subset\p^{15}$ by Corollary \ref{cor:L_z}(iv). Therefore $X\subset\p^{16}$ is a degenerate minimal embedding of $S_4\subset\p^{15}$.

Assume now $b=4$, $\delta=0$ and $n=7$. Let $X_p\subset {\mathcal L}_z=S_4\subset\p^{15}$ (cf. Notation \ref{not:X_p}), and let $p'\in\p^{15}$ be a general point. The entry locus of $p'$ (relative to ${\mathcal L}_z$) is a $6$-dimensional quadric hypersurface $\Sigma_{p'}\subset\p_{p'}^7$ intersecting $X_p$ in a $3$-dimensional subvariety. Let $\p_{pp'}^8:=\langle p,\p_{p'}^7\rangle$. Then $\Phi:\p^{16}\da E_6$ restricted to $\p_{pp'}^8\subset\p^{16}$ is a birational transformation onto a smooth $8$-dimensional quadric hypersurface $\Phi(\p_{pp'}^8)\subset E_6$, namely a fibre of the tangential projection of $E_6$ onto an $8$-dimensional quadric hypersurface (see Remark \ref{rem:tangential}). Repeat the construction taking a general $p''\in\p^{15}$ such that $\Sigma_{p'}\cap\Sigma_{p''}$ is a $\p^3$ (note that this is always possible). Then $\p_{pp'}^8\cap\p_{pp''}^8$ is a linear $\p^4$ and hence $\Phi(\p^4)\subset E_6$ is also a linear $\p^4$, namely the intersection of the $8$-dimensional quadric hypersurfaces $\Phi(\p_{pp'}^8)$ and $\Phi(\p_{pp''}^8)$ of $E_6$.  As $\Sigma_{p'}\cap\Sigma_{p''}\cap X_p$ is a finite subscheme, we get a contradiction as in Theorem \ref{thm:spinor}.
\end{proof}

\begin{remark}\label{rem:cor}
We point out that the Grassmannians of lines, the $10$-dimensional spinor variety $S_4$ and the $E_6$-variety are the only secant defective irreducible Hermitian symmetric spaces different from projective spaces and quadric hypersurfaces. As the image of a special birational transformation $\Phi:\p^r\da Z$ of type $(2,b)$ is secant defective (see Lemma \ref{lem:bound}), Theorems \ref{thm:grassmannians}, \ref{thm:spinor} and \ref{thm:cartan} give the classification of special birational transformations $\Phi:\p^r\da Z$ of type $(2,b)$ onto an irreducible Hermitian symmetric space $Z$ different from a projective space and a quadric hypersurface. According to Propositions \ref{prop:grassmannians}, \ref{prop:spinor} and \ref{prop:cartan}, it is natural to classify the special birational transformations of type $(2,b)$ onto codimension-$s$ linear sections of the above rational homogeneous varieties. In order to get shorter proofs, we will obtain this classification in Corollary \ref{cor:homogeneous} as a consequence of our more detailed study of the $(2,2)$ case.
\end{remark}

\section{The quadro-quadric case}\label{sect:main}

Let $\Phi:\p^r\da\p^r$ be a special Cremona transformation of type $(2,b)$. Then we get from Proposition \ref{prop:num} that $n=b(\delta+1)-2$. Therefore $$\delta(\p^r)=r+1=2n+3-\delta=2b-1+(2b-1)\delta$$

On the other hand, for special birational transformations $\Phi:\p^r\da Z$ of type $(2,b)$ with $Z\neq\p^r$ we get the following:

\begin{lemma}\label{lem:bound}
Let $\Phi:\p^r\da Z$ be a special birational transformation of type $(2,b)$. If $Z\neq\p^r$ then $\delta(Z)\geq 2b+(2b-1)\delta$, and equality holds if and only if $Z$ is an LQEL-manifold.
\end{lemma}

\begin{proof}
As $Z$ is a conic-connected manifold, it follows from \cite[Proposition 3.2(i)-(ii)]{i-r3} that $2i(Z)\leq \dim(Z)+\delta(Z)$ with equality if and only if $Z$ is an LQEL-manifold. So we conclude in view of Corollary \ref{cor:sec}(i) and Proposition \ref{prop:num}.
\end{proof}

\begin{remark}\label{rem:LQEL}
From this point of view, the next natural case to consider after the special Cremona transformations is the extremal case in which $\delta(Z)=2b+(2b-1)\delta$, or equivalently, the case in which $Z$ is an LQEL-manifold. We recall that $Z$ is said to be an LQEL-manifold if there exists a $\delta(Z)$-dimensional quadric hypersurface passing through any two general points $z,z'\in Z$ (see \cite[Definition 1.1]{russo}). Due to the lack of examples, secant defective LQEL-manifolds are expected to be nothing but non-linearly normal secant defective QEL-manifolds.
\end{remark}

Let us recall the following result of Ein and Shepherd-Barron:

\begin{theorem}\label{thm:e-sb}\cite[Theorem 2.6]{e-sb}
If $\Phi:\p^r\da\p^r$ is a special Cremona transformation of type $(2,2)$ then $X\subset\p^r$ is a Severi variety. 
\end{theorem}

To the best of the authors' knowledge, the only examples of special birational transformations $\Phi:\p^r\da Z$ of type $(2,2)$ onto a prime Fano manifold $Z\neq\p^r$ are the following:

\begin{example}[see \cite{semple}]\label{ex:semple}
Let $X\subset\p^{2n+2}$ be a rational normal scroll of degree $n+3$. The quadric hypersurfaces passing through $X$ define a special birational transformation of type $(2,2)$ onto a Grassmannian of lines $\g(1,n+2)\subset\p^{(n^2+5n+4)/2}$.
\end{example}

\begin{example}\label{ex:a-s}
Let $\Phi:\p^r\da\p^r$ be a special Cremona transformation as in Theorem \ref{thm:e-sb}. The restriction to a general hyperplane $H\subset\p^r$ gives a special birational transformation $\Phi:\p^{r-1}\da Z\subset\p^r$ of type $(2,2)$ onto a smooth quadric hypersurface of $\p^r$.
\end{example}

The main result of the paper, generalizing Theorem \ref{thm:e-sb} as suggested by Remark \ref{rem:LQEL}, is the following:

\begin{theorem}\label{thm:(2,2)}
Let $\Phi:\p^r\da Z$ be a special birational transformation of type $(2,2)$. If $Z\neq\p^r$ then $\delta(Z)\geq 4+3\delta$, and equality holds (or equivalently, and $Z$ is a LQEL-manifold) if and only if one of the following holds:
\begin{enumerate}
\item[(i)] $X\subset\p^{2n+2}$ is a rational normal scroll of degree $n+3$ and $Z=\g(1,n+2)$;
\item[(ii)] $X\subset\p^r$ is a hyperplane section of a Severi variety and $Z\subset\p^{r+1}$ is a quadric hypersurface.
\end{enumerate}
\end{theorem}

\begin{remark}\label{rem:key}
If $\Phi:\p^r\da Z$ is a special birational transformation of type $(2,2)$ and $\delta>0$ then ${\mathcal L}_z\subset\p^{r-1}$ is a prime Fano manifold covered by lines by Corollary \ref{cor:L_z}(vii). For a general $z\in Z$, the family of lines contained in ${\mathcal L}_z\subset\p^{r-1}$ corresponds to the family of planes in $Z$ passing through $z$. The key point in the proof of Theorem \ref{thm:(2,2)} is the computation of the index of ${\mathcal L}_z$ given in Corollary \ref{cor:index L_Z}. To this aim, we have to study the family of planes contained in $Z$.
\end{remark}

\subsection{The family of planes in $Z$}

If $\Phi:\p^r\da Z$ is a special birational transformation of type $(2,b)$ and $b\geq 2$ then $Z$ is covered by planes (see Corollary \ref{cor:L_z}(vii)). Moreover, if $b=2$ and $\p^2\subset Z$ is a plane not contained in $Y\subset Z$ and not contracted by $\Psi:=\Phi^{-1}$ then a case-by-case analysis gives the following:

\begin{lemma}\label{lem:case-by-case}
Let $\Phi:\p^r\da Z$ be a special birational transformation of type $(2,2)$, and let $\p^2\subset Z$ be a plane not contained in $Y\subset Z$ and not contracted by $\Psi$. Then one of the following holds:
\begin{enumerate}
\item[(i)] $\p^2\cap Y$ is a line and $\Psi(\p^2)$ is a plane intersecting $X\subset\p^r$ in a line;
\item[(ii)] $\length(\p^2\cap Y)=3$ and $\Psi(\p^2)$ is a plane such that $\length(\Psi(\p^2)\cap X)=3$;
\item[(iii)] $\length(\p^2\cap Y)=2$ and $\Psi(\p^2)$ is a quadric in $\p^3$ intersecting $X\subset\p^r$ in a conic and a point off the conic;
\item[(iv)] $\length(\p^2\cap Y)=1$ and $\Psi(\p^2)$ is a smooth rational normal scroll in $\p^4$ intersecting $X\subset\p^r$ in a smooth rational quartic curve;
\item[(v)] $\p^2\cap Y=\emptyset$ and $\Psi(\p^2)$ is a Veronese surface in $\p^5$ intersecting $X\subset\p^r$ in a sextic curve of arithmetic genus $p_a=1$.
\end{enumerate}
\end{lemma}

Now we deduce the existence of smooth rational quartic and elliptic sextic curves on a secant defective LQEL-manifold from the existence of conics, and we compute the dimension of these families of curves. We will use the standard terminology from the theory of deformations of rational curves in what follows (see for instance \cite{kollar} and \cite{debarre}):

\begin{proposition}\label{prop:deformation}
If $X\subset\p^r$ is an LQEL-manifold of secant defect $\delta>0$ then:
\begin{enumerate}
\item[(i)] There exists an irreducible family $\mathcal C$ of $1$-free conics on $X$ of dimension $2n+\delta-3$ whose general member $C\in\mathcal C$ is smooth;
\item[(ii)] There exists an irreducible family $\mathcal Q$ of $2$-free rational quartic curves on $X$ of dimension $3n+2\delta-3$ whose general member $Q\in\mathcal Q$ is smooth;
\item[(iii)] There exists an irreducible family $\mathcal E$ of $2$-free elliptic sextic curves on $X$ of dimension $3n+3\delta$ whose general member $E\in\mathcal E$ is smooth.
\end{enumerate}
\end{proposition}

\begin{proof} Part (i) is well known (see for instance \cite[Theorem 2.1(2)]{russo}). Let $C_1\cup C_2$ be the rational tree on $X$ consisting of the union of two $1$-free conics $C_1$ and $C_2$ meeting at a point. Since $C_i$ are $1$-free we deduce from \cite[Proposition 4.24]{debarre}, by taking $r_1 =2$ and $r_2=1$, that $C_1\cup C_2$ is smoothable into a $2$-free rational quartic curve $Q\subset X$. As $-K_X\cdot C_i=n+\delta$ (see for instance \cite[Theorem 2.1(5)]{russo}), we deduce $-K_X\cdot Q=2n+2\delta$. Then the family of deformations of $Q$ in $X$ has dimension $\dim(\mor(Q,X))-\dim(\aut(Q))=-K_X\cdot Q+\dim(X)-3=3n+2\delta-3$ (see for instance \cite[Theorem 2.6 and 2.11]{debarre}). This gives (ii). Now let $C\cup Q$ be the sextic curve on $X$ with $p_a=1$ consisting of the union of a $1$-free conic $C$ and a $2$-free quartic rational curve $Q$ meeting at two points, which is smoothable inside a Veronese surface. Then we can argue exactly as in \cite[Proposition 4.24]{debarre} to deduce that $C\cup Q$ is smoothable into a $2$-free sextic elliptic curve $E\subset X$ (take $r_C=2$ and $r_Q=1$ and note that the key point is that $H^1(C,T_X\otimes\O_{\p^1}(-2))=H^1(Q, T_X\otimes\O_{\p^1}(-3))=0$). Then the family of deformations of $E$ in $X$ has dimension $\dim(\mor(E,X))-\dim(\aut(E))=-K_X\cdot E-1=3n+3\delta-1$. Since every isomorphism class of elliptic curves appears in the smoothing of $C\cup Q$, we get (iii).
\end{proof}


\begin{remark}
Proposition \ref{prop:deformation} is sharp, as the following example shows. For $n\geq 3$, let $X\subset\p^r$ be (a linear projection of) the $2$-Veronese embedding of $\p^n$. Then the family of rational quartic curves (resp. elliptic sextic curves) in $X$, corresponding to the family of conics (resp. plane cubic curves) in $\p^n$, is $2$-free but not $3$-free.
\end{remark}

\begin{corollary}\label{cor:planes}
Let $\Phi:\p^r\da Z$ be a special birational transformation of type $(2,2)$, and let $\mathcal P_{\lambda}$ denote the family of planes in $Z$ intersecting $Y\subset Z$ in a finite subscheme of length $\lambda\in\{0,1,2,3\}$. Then $\dim(\mathcal P_3)=3n$, and if $\delta>0$ then $\dim(\mathcal P_{\lambda})\leq 3n+(3-\lambda)\delta$ for $\lambda\in\{0,1,2\}$.
\end{corollary}

\begin{proof}
It is obvious from Lemma \ref{lem:case-by-case}(ii) that $\dim(\mathcal P_3)=3n$, so we assume $\delta>0$ and we start with $\mathcal P_2$. According to Lemma \ref{lem:case-by-case}(iii), elements of $\mathcal P_2$ correspond to quadrics in $\p^3$ intersecting $X\subset\p^r$ in a conic and a point off the conic. The dimension of the family of subschemes of $X$ consisting of a conic and a point off the conic has dimension $3n+\delta-3$ by Proposition \ref{prop:deformation}(i). So we deduce that $\dim(\mathcal P_2)\leq 3n+\delta$ by observing that the family of quadrics in $\p^3$ containing a fixed subscheme consisting of a conic and a point off the conic has dimension $3$.

Now we study $\mathcal P_1$. Since elements of $\mathcal P_1$ correspond to smooth rational normal scrolls in $\p^4$ intersecting $X\subset\p^r$ in a smooth rational quartic curve by Lemma \ref{lem:case-by-case}(iv), and $\dim(\mathcal Q)=3n+2\delta-3$ by Proposition \ref{prop:deformation}(ii), it is enough to show that given a rational normal quartic $Q\subset\p^4$ the family of smooth rational normal scrolls in $\p^4$ containing $Q$ has dimension $3$. A smooth rational normal scroll $S_{1,2}\subset\p^4$ is isomorphic to $\mathbb F_1$ embedded by $|C_0+2f|$. Then $Q=C_0+3f$, as $Q\cdot f=1$ and $Q\cdot C_0=2$ since the $(-1)$-line is the only secant line to $Q$ in $S_{1,2}$ corresponding to the unique point in $Y$ of the element of $\mathcal P_1$. Therefore every $S_{1,2}$ containing $Q$ corresponds to a secant line $L$ of $Q$, which plays the role of $C_0$, and an isomorphism $L\cong Q$ keeping fixed the two points $L\cap Q$. So the family of rational normal scrolls in $\p^4$ containing $Q$ has dimension $3$, and hence $\dim(\mathcal P_1)\leq 3n+2\delta$.

Finally, we consider $\mathcal P_0$. Since elements of $\mathcal P_0$ correspond to Veronese surfaces in $\p^5$ intersecting $X\subset\p^r$ in a sextic curve of $p_a=1$ by Lemma \ref{lem:case-by-case}(v), and $\dim(\mathcal E)=3n+3\delta$ by Proposition \ref{prop:deformation}(iii), it is enough to show that given an elliptic normal sextic $E\subset\p^5$ the family of Veronese surfaces in $\p^5$ containing $E$ is finite. The dimension of the Hilbert scheme of elliptic sextic curves (resp. Veronese surfaces) in $\p^5$ is $36$ (resp. $27$) and every Veronese surface contains a $\p^9$ of elliptic sextic curves, so it is enough to show that an elliptic sextic curve $E$ is contained in a Veronese surface, and this is done by taking a plane cubic curve isomorphic to $E$ and considering the Veronese embedding of the corresponding plane. Therefore $\dim(\mathcal P_0)\leq 3n+3\delta$.
\end{proof}

\begin{remark}\label{rem:planes}
In Corollary \ref{cor:planes} one can actually prove that furthermore $\dim(\mathcal P_{\lambda})=3n+(3-\lambda)\delta$ and that $\mathcal P_{\lambda}$ is a covering family of planes in $Z$ for $\lambda\in\{1,2,3\}$, but it seems more difficult to prove the same for $\lambda=0$. On the other hand, the picture looks different when $\delta=0$. For instance, if $X\subset\p^4$ is a rational quartic curve and $\Phi:\p^4\da Z$ is the corresponding special birational transformation of type $(2,2)$ onto a quadric hypersurface of $\p^5$ then $Y\subset Z$ is a Veronese surface, $\dim(\mathcal P_1)=\dim(\mathcal P_3)=3$, and $\mathcal P_0=\mathcal P_2=\emptyset$.
\end{remark}

Now we can estimate the index of $\mathcal L_z$ when $\delta>0$. We know from Corollary \ref{cor:L_z}(vii) that $\mathcal L_z$ is a prime Fano manifold which is covered by lines. Let $\mathcal M^z$ be a covering family of lines in $\mathcal L_z$ of maximal dimension, and let $\lambda_Z:=\min\{\lambda\mid \mathcal P_{\lambda}\, \text{is a covering family of planes in $Z$}\}$ (cf. Remark \ref{rem:planes}).

\begin{corollary}\label{cor:index L_Z}
Let $\Phi:\p^r\da Z$ be a special birational transformation of type $(2,2)$. If $\delta>0$ then $\dim(\mathcal M^z)\leq n+(4-\lambda_Z)\delta$ for general $z\in Z$. Equivalently, $i({\mathcal L}_z)\leq 2+(3-\lambda_Z)\delta$ for general $z\in Z$.
\end{corollary}

\begin{proof}
Let $z\in Z$ and $l\in\mathcal L_z$ be general points. Let $L\subset Z$ be the line corresponding to $l$. As $L\cap Y=\emptyset$ then $\mathcal M^z$ corresponds to the members of $\mathcal P_{\lambda_Z}$ passing through $z$, that is, the family of planes intersecting $Y\subset Z$ in a line (see Lemma \ref{lem:case-by-case}(i)) does not produce a covering family of lines of $\mathcal L_z$. Therefore $\dim(\mathcal M^z)=\dim(\mathcal P_{\lambda_Z})+2-\dim(Z)\leq 3n+(3-\lambda_Z)\delta+2-(2n+2-\delta)=n+(4-\lambda_Z)\delta$. Let $\mathcal M^z_l\subset\mathcal M^z$ denote the variety of lines passing through $l\in\mathcal L_z$, so $\dim(\mathcal M^z_l)=\dim(\mathcal M^z)+1-\dim(\mathcal L_z)\leq n+(4-\lambda_Z)\delta+1-(n+1+\delta)=(3-\lambda_Z)\delta$. As $i(\mathcal L_z)=2+\dim(\mathcal M^z_l)$, we deduce $i(\mathcal L_z)\leq 2+(3-\lambda_Z)\delta$.
\end{proof}

\subsection{Proof of the main result}

We are now in position to prove Theorem \ref{thm:(2,2)}. The proof is divided into two cases, namely $\delta=0$ and $\delta>0$:

\begin{proposition}\label{prop:delta=0}
Let $\Phi:\p^r\da Z$ be a special birational transformation of type $(2,2)$. If $\delta=0$ and $\delta(Z)=4$ (or equivalently, and $Z$ is an LQEL-manifold) then $X\subset\p^{2n+2}$ is a rational normal scroll of degree $n+3$ and $Z=\g(1,n+2)$.
\end{proposition}

\begin{proof}
Let $z\in Z$ be a general point. Consider ${\mathcal L}_z\subset\p^{2n+1}$ the variety of lines passing through $z$. We claim that ${\mathcal L}_z$ is the Segre embedding of $\p^1\times\p^n$. Since we assume $\delta(Z)=4$ (or equivalently, since we assume that $Z$ is an LQEL-manifold) we get that ${\mathcal L}_z\subset\p^{2n+1}$ is a QEL-manifold (see for instance \cite[Theorem 2.3(3)]{russo}) of dimension $n+1$ with $\delta({\mathcal L}_z)=2$ by Corollary \ref{cor:L_z}. According to \cite[Theorem 2.2]{i-r2}, it is enough to show that ${\mathcal L}_z$ is not a prime Fano manifold. To get a contradiction, assume $\pic({\mathcal L}_z)=\z\langle H_{{\mathcal L}_z}\rangle$. Let $u\in\p^{2n+1}$ be a general point and let $\Sigma_u\subset\p_u^3$ denote the corresponding entry locus ($\Sigma_u$ is a $2$-dimensional quadric). Arguing as in Theorem \ref{thm:grassmannians}, and having in mind that $X_p\subset{\mathcal L}_z$ is a divisor, we get that $X_p\cap\Sigma_u$ is a $(1,3)$-curve, whence contradicting the assumption on $\pic(\mathcal L_z)$. Therefore ${\mathcal L}_z=\p^1\times\p^n\subset\p^{2n+1}$, and we conclude as in Theorem \ref{thm:grassmannians} that $X\subset\p^{2n+2}$ is a rational normal scroll of degree $n+3$. So, in particular, $Z=\g(1,n+2)$.
\end{proof}

\begin{proposition}\label{prop:main}
Let $\Phi:\p^r\da Z$ be a special birational transformation of type $(2,2)$. If $\delta>0$ and $\delta(Z)=4+3\delta$ (or equivalently, and $Z$ is an LQEL-manifold) then $X\subset\p^r$ is a hyperplane section of a Severi variety and $Z\subset\p^{r+1}$ is a quadric hypersurface.
\end{proposition}

\begin{proof}
Since $\delta(Z)=4+3\delta$ (or equivalently, since $Z$ is an LQEL-manifold) we deduce that ${\mathcal L}_z\subset\p^{2n+1-\delta}$ is a prime Fano QEL-manifold such that $\dim({\mathcal L}_z)=n+1+\delta$ and $\delta({\mathcal L}_z)=2+3\delta$ by Corollary \ref{cor:L_z}. So $i({\mathcal L}_z)=(\dim({\mathcal L}_z)+\delta({\mathcal L}_z))/2=(n+3+4\delta)/2$ and we deduce from Corollary \ref{cor:index L_Z} that $\delta\geq (n-1)/2$. On the other hand, as $SX\neq\p^r$ we get from \cite[Ch. II, Corollary 2.11]{zak2} that $\delta\leq n/2$, and equality holds if and only if $X\subset\p^r$ is a Severi variety. In that case $Z=\p^r$ (cf. \cite[Theorem 2.6]{e-sb}), so we get $\delta=(n-1)/2$. Therefore $X\subset\p^r$ is expected to be a hyperplane section of a Severi variety according to \cite[Ch. IV, Remark 5.16]{zak2}. We provide a proof of this statement combining some ideas of \cite{zak2} and \cite{russo} in our particular setting. Since $\deg(SX)=3$ we deduce from \cite[Theorem 2.2]{i-r2} that either $X\subset\p^r$ is a prime Fano manifold of index $i(X)=(n+\delta)/2$, or else $n=3$ and $X\subset\p^7$ is a hyperplane section of the Severi variety $\p^2\times\p^2$. In the first case $n\equiv3\pmod{4}$, so we can assume $n\geq 7$ by \cite{fuj2} and hence $\delta\geq 3$. Therefore we deduce from the Divisibility Property \cite[Theorem 2.8(2)]{russo} that $2^{(n-3)/4}$ divides $(n+1)/2$, and hence $(n,\delta)\in\{(7,3),(15,7)\}$.

If $(n,\delta)=(7,3)$ then $i(X)=5$, so we deduce from Mukai's classification of Fano manifolds of coindex $3$ that $X\subset\p^{13}$ is a hyperplane section of the Severi variety $\g(1,5)\subset\p^{14}$ (see \cite{muk}).

If $(n,\delta)=(15,7)$ we argue in a different way. Let ${\mathcal L}_x\subset\p^{14}$ denote the variety of lines passing through a general point $x\in X$. Then $\dim({\mathcal L}_x)=i(X)-2=(n+\delta)/2-2=9$ and $\delta({\mathcal L}_x)=5$, whence $i({\mathcal L}_x)=(9+5)/2=7$ and hence ${\mathcal L}_x\subset\p^{14}$ is a hyperplane section of the spinor variety $S_4\subset\p^{15}$ by Mukai's classification. Let us show that $X\subset\p^{25}$ is a hyperplane section of the Severi variety $E_6\subset\p^{26}$. Now we essentially argue as in \cite[Ch. IV, \S 4]{zak2}. For general $u\in SX$, let $\Sigma_u\subset\p_u^8$ be the entry locus of $u$. Then $\Sigma_u$ is a $7$-dimensional quadric by Proposition \ref{prop:sec}(ii). Furthermore, $\Sigma_u$ is smooth (see \cite[Proposition 2.4]{ohno1}; see also \cite{f-r}). Consider the linear projection $\p^{25}\da\p^{16}$ from $\p_u^8$ and the restriction $\pi:X\da\pi(X)$. We claim that every fibre of $\pi$ is linear. Let $x,x'\in X$ be any two points such that $\pi(x)=\pi(x')$, and let $\langle x,x'\rangle$ be the line joining $x$ and $x'$. Consider the point $p_u(xx'):=\langle x,x'\rangle\cap\p_u^8$. If $p_u(xx')\notin\Sigma_u$ then $x,x'\in\Sigma_u$ (cf. Remark \ref{rem:entry}), giving a contradiction. Therefore $p_u(xx')\in\Sigma_u$, and hence $\langle x,x'\rangle\subset X$ since $X\subset\p^{25}$ is defined by quadric hypersurfaces. This proves the claim. Note that $T_uSX\subset\p^{25}$ is a hyperplane such that $T_xX\subset T_uSX$ for every $x\in\Sigma_u$ by Terracini's lemma. In particular, if $x,x'\in X$ are two points such that $\pi(x)=\pi(x')$ then the above argument shows that $x,x'\in T_{p_u(xx')}X\subset T_uSX$, and hence we deduce that $\pi:X\da\pi(X)$ is an isomorphism off $X\cap T_uSX$. Let $B\subset\pi(X)$ denote the base locus of $\pi^{-1}$. As $\pi^{-1}(B)\subset T_uSX$, we deduce that $B\subset\p^{16}$ is contained in a hyperplane. We will prove in what follows that $B$ is actually a $10$-dimensional spinor variety in $\p^{15}$. As $\pi^{-1}(b)$ is a linear space for every $b\in B$ and $\pi^{-1}(b)\cap\Sigma_u$ is also linear, we deduce that $\dim(\pi^{-1}(b))\leq4$ for every $b\in B$ since $\Sigma_u$ does not contain any linear space of dimension bigger that $3$. Consider the incidence variety $I:=\{(x,y)\mid \langle x,y \rangle\subset X\}\subset\Sigma_u\times\pi^{-1}(B)$. Since $\dim({\mathcal L}_x)=9$ for general $x\in\Sigma_u$, $\dim(I)=17$ and $\dim(\pi^{-1}(B))\leq 14$, we deduce that $\dim(\pi^{-1}(b))\geq 4$ for every $b\in B$. Therefore $\dim(\pi^{-1}(B))=14$, $\dim(\pi^{-1}(b))=4$ for every $b\in B$ and $\dim(B)=10$. Furthermore $\pi^{-1}(B)=X\cap T_uSX$, as $X\cap T_uSX$ is irreducible since $\pic(X)$ is generated by the hyperplane section. Let $S_u\subset\g(3,8)$ denote the family of $\p^3$'s contained in $\Sigma_u$, which is a $10$-dimensional spinor variety. Therefore we get a morphism $\beta:B\to S_u$ defined by $\beta(b)=\pi^{-1}(b)\cap\Sigma_u$. We claim that $\beta$ is an isomorphism. Let us show first that $\beta$ is a birational morphism. Let $x\in\Sigma_u$ be a general point and let $\p_x^3\in S_u$ be a general element containing $x$. Then $\p_x^3$ corresponds to a general $\p^2$ contained in ${\mathcal L}_x=S_4\cap H\subset\p^{14}$, and hence it is contained in a unique $\p^3$ of ${\mathcal L}_x$, namely a contact $\p^3$ of $S_4\cap H$. Therefore we get a unique $\p^4\subset X$ containing $\p_x^3$. Hence the projection from $\p_u^8$ of this $\p^4$ gives the inverse map $\beta^{-1}:S_u\da B$. Finally, we show that $\beta^{-1}$ is given by the generator of $\pic(S_u)$. The generator of $\pic(S_u)$ corresponding to the minimal embedding is given by the ``square root" of the Pl\"ucker embedding of $S_u\subset\g(3,8)$. Therefore a generator of $N_1(S_u)$ corresponds to a $4$-dimensional quadric cone $\Gamma\subset\Sigma_u$ of vertex a plane $\p^2_{\Gamma}$. Let us choose a general point $x\in\p^2_{\Gamma}$ and let us project $\Gamma$ from $x$ into $\mathcal L_x=S_4\cap H\subset\p^{14}$, getting a $3$-dimensional quadric cone $\Delta\subset\mathcal L_x$ of vertex a line $\p^1_{\Delta}$. Fixing again a general point $w\in\p^1_{\Delta}$, let us project $\Delta$ from $w$ into $\mathcal L_w =\g(1,4)\cap H\subset\p^8$, getting a $2$-dimensional quadric cone $\Lambda\subset\mathcal L_w$ of vertex a point, say $l$. The cone $\Lambda$ is contained in the Schubert cycle $\Omega(L,\p^4)$ of lines meeting the line $L\subset\p^4$ ($l\in\g(1,4)$ is a point and $L\subset\p^4$ denotes the corresponding line), so $\Omega(L,\p^4)\cap H$ gives a $\p^5$ in $\p^8$ containing $\Lambda$ and intersecting $\g(1,4)\cap H$ in $\Omega(L,\p^4)\cap H$ (this $\p^5$ is nothing but $T_l\g(1,4)\cap H$, and we recall that $\Omega(L,\p^4)=\g(1,4)\cap T_l\g(1,4)$ is a cone over the Segre embedding $\p^1\times\p^2$ of vertex $l$). Going back to $S_4\cap H$, this $\p^5$ corresponds to a $\p^6$ containing every (contact) $\p^3$ determined by a plane of $\Delta$. So going back again to $X$ we get a $\p^7$ containing every $\p^4$ determined by a $\p^3$ of $\Gamma$. Therefore, as $\p^7\cap\p_u^8$ is the linear span of $\Gamma$, we deduce that $\beta^{-1}$ sends the generator of $N_1(S_u)$ to a line of $B$ and hence $\beta^{-1}:S_u\da B$ is given by the generator of $\pic(S_u)$. Since $B\subset\p^{15}$ is non-degenerate, this proves that $\beta$ is an isomorphism, as claimed. We now prove that $\pi^{-1}:\pi(X)\da X$ is defined by quadric hypersurfaces. To this aim, it is enough to show that a line in $\pi(X)$ not intersecting $B$ is mapped by $\pi^{-1}$ onto a conic in $X$ intersecting $\Sigma_u$ in one point only. We actually prove the converse. Let $C\subset X$ be a conic not contained in the hyperplane $T_uSX$, and passing through a point $x\in\Sigma_u$. If $\length(C\cap\Sigma_u)\geq 2$ then either $C\subset\Sigma_u$, or else $\pi(C)$ is a point and hence $C\subset T_uSX$, so we get $C\cap\Sigma_u=x$. As $T_xX\subset T_uSX$, the plane spanned by $C$ intersect $T_uSX$ in the line $T_xC$. In particular, $C\cap\pi^{-1}(B)=\emptyset$ and hence $\pi(C)$ is a line not intersecting $B$, as required. On the other hand, the quadric hypersurfaces defining the degenerate embedding of $B\subset\p^{16}$ give a birational map $\p^{16}\da E_6\subset\p^{26}$ and hence we immediately get that $\pi(X)\subset\p^{16}$ is a quadric hypersurface containing $B$ and that $X\subset\p^{25}$ is a hyperplane section of the Severi variety $E_6\subset\p^{26}$.
\end{proof}

\begin{proof}[Proof of Theorem \ref{thm:(2,2)}]
If follows from Lemma \ref{lem:bound} and Propositions \ref{prop:delta=0} and \ref{prop:main}.
\end{proof}

\begin{remark}
Some partial results on secant defective manifolds with almost maximal defect $\delta=(n-1)/2$, following the ideas of \cite{f-r} and \cite{tango}, can be found in \cite{ohno1}.
\end{remark}

\begin{remark}
A weaker statement than that of Proposition \ref{prop:main} has been independently proved in a different way in \cite[Theorem 5.1]{stagliano}.
\end{remark}

Now we can state and prove the result announced in Remark \ref{rem:cor}. First we need an auxiliary lemma:

\begin{definition}
We say that $X\subset\p^r$ is a \emph{strong QEL-manifold} if $\Sigma_u\subset\p_u^{\delta+1}$ is a quadric hypersurface \emph{for every $u\in SX$}.
\end{definition}

\begin{lemma}\label{lem:strong}
Let $\Phi:\p^r\da Z$ be a special birational transformation of type $(2,b)$. If $Y\subset Z$ is smooth then $X\subset\p^r$ is a strong QEL-manifold.
\end{lemma}

\begin{proof}
It follows from \cite[Theorem 1.1]{e-sb} that $W$, which is the blowing-up of $\p^r$ along $X$ by definition, is also the blowing-up of $Z$ along $Y$. Therefore, for every $u\in\sigma(E_Z)=SX$ the entry locus $\Sigma_u$ is a quadric hypersurface of dimension $\delta$ (cf. Corollary \ref{cor:sec}(ii)).
\end{proof}

\begin{remark}
If $Y\subset Z$ is singular then $X\subset\p^r$ may not be a strong QEL-manifold. For instance, this is the case of Example \ref{ex:semple} with $n\geq 2$.
\end{remark}

\begin{corollary}\label{cor:homogeneous}
If $\Phi:\p^r\da Z$ is a special birational transformation of type $(2,b)$ onto (a linear section of) a rational homogeneous variety different from a projective space and a quadric hypersurface then one of the following holds:
\begin{enumerate}
\item[(i)] $X\subset\p^{6-s}$ is a codimension-$s$ linear section of a degenerate Segre embedding of $\p^1\times\p^2$ and $Z$ is a codimension-$s$ linear section of $\g(1,4)$ for $s\in\{0,1,2\}$;
\item[(ii)] $X\subset\p^{2n}$ is a degenerate Segre embedding of $\p^1\times\p^{n-1}$ and $Z=\g(1,n+1)$;
\item[(iii)] $X\subset\p^{10-s}$ is a codimension-$s$ linear section of a degenerate Pl\"ucker embedding of $\g(1,4)$ and $Z$ is a codimension-$s$ linear section of $S_4$ for $s\in\{0,1,2,3\}$;
\item[(iv)] $X\subset\p^{16}$ is a degenerate minimal embedding of $S_4$ and $Z=E_6$;
\item[(v)] $X\subset\p^{2n+2}$ is a rational normal scroll of degree $n+3$ and $Z=\g(1,n+2)$.
\end{enumerate}
\end{corollary}

\begin{proof}
As $\delta(Z)\geq 2$ by Lemma \ref{lem:bound}, we deduce from \cite{kaji} that $Z$ is a codimension-$s$ linear section of either a Grassmannian of lines, or the $10$-dimensional spinor variety $S_4$, or the $E_6$-variety. Then we apply Propositions \ref{prop:grassmannians}, \ref{prop:spinor} and \ref{prop:cartan}, respectively.

If $Z$ is a a codimension-$s$ linear section of $\g(1,q+1)$ then either $b=1$, $\delta=2-s$, $s\in\{0,1,2\}$ and $n=q-s$, or $b=2$, $\delta=0$, $s=0$ and $n=q-1$. In the first case, $X\subset\p^r$ is degenerate (see Corollary \ref{cor:L_z}(iv)) and $\mathcal L_z\subset\p^{r-1}$ is projectively equivalent to $X\subset\p^{r-1}$. Since $\mathcal L_z\subset\p^{r-1}$ is a codimension-$s$ linear section of $\p^1\times\p^{q-1}$, the same is true for $X\subset\p^{r-1}$ and hence $Y=\g(1,q-1)\subset Z$ (cf. \cite[Ch. III, Theorem 3.8]{zak2}). In particular, $Y$ is smooth and we deduce from Lemma \ref{lem:strong} that $X\subset\p^{r-1}$ is a strong QEL-manifold. Having in mind that every entry locus of $X\subset\p^{r-1}$ is a quadric hypersurface (namely a linear section of the corresponding entry locus of $\p^1\times\p^{q-1}$), we get that either $s=0$ (as in Theorem \ref{thm:grassmannians}), or else one immediately deduces $q=3$ and $s\in\{1,2\}$. This gives cases (i) and (ii). On the other hand, in the second case we get case (v) by Theorem \ref{thm:grassmannians}.

If $Z$ is a codimension-$s$ linear section of $S_4$ then either $b=1$, $\delta=4-s$, $s\in\{0,1,2,3,4\}$ and $n=6-s$, or $b=2$, $\delta=0$, $s=2$ and $n=3$, or $b=3$, $\delta=0$, $s=0$ and $n=4$. In the first case, $\mathcal L_z\subset\p^{9-s}$ is a codimension-$s$ linear section of $\g(1,4)$ and hence the same is true for $X\subset\p^{9-s}$. Therefore $Y=\p^4\subset Z$ (cf. \cite[Ch. III, Theorem 3.8]{zak2}) and we deduce from Lemma \ref{lem:strong} that $X\subset\p^{9-s}$ is a strong QEL-manifold. Therefore either $s=0$ (as in Theorem \ref{thm:spinor}), or else we claim that $s\in\{1,2,3\}$. Actually, a codimension-$s$ linear section of $\g(1,4)$ has Picard group generated by the hyperplane section for $s\in\{1,2,3\}$ and hence it does not contain a codimension-$1$ quadric hypersurface. On the other hand, if $s=4$ then $X\subset\p^5$ is a Del Pezzo surface, which is not a strong QEL-manifold as it is swept out by conics. This proves the claim and gives (iii). The second case is ruled out by Theorem \ref{thm:(2,2)}, and the third case was ruled out in Theorem \ref{thm:spinor}.

Finally, if $Z$ is a codimension-$s$ linear section of $E_6$ then either $b=1$, $\delta=6-s$, $s\in\{0,1,2,3,4,5,6\}$ and $n=10-s$, or $b=2$, $\delta=0$, $s=4$ and $n=5$, or $b=2$, $\delta=1$, $s=1$ and $n=7$, or $b=3$, $\delta=0$, $s=2$ and $n=6$, or $b=4$, $\delta=0$, $s=0$ and $n=7$. In the first case, arguing as before one gets that $X\subset\p^{15-s}$ is a codimension-$s$ linear section of $S_4$ for $s\in\{0,1,2,3,4,5,6\}$ and $Y\subset Z$ is an $8$-dimensional quadric hypersurface (cf. \cite[Ch. III, Theorem 3.8]{zak2}). In particular, $X\subset\p^{15-s}$ is a strong QEL-manifold. We claim that this happens only for $s=0$. To prove the claim, it is enough to show that $X\subset\p^{15-s}$ is not a strong QEL-manifold for $s=1$. Consider a spinor variety $S_4\subset\p^{15}$. The family of $6$-dimensional quadric hypersurfaces in $S_4$ is an $8$-dimensional quadric hypersurface $Q^8$. Then there exists a globally generated rank-$8$ vector bundle $E$ on $Q^8$ giving an embedding $Q^8\to\g(7,15)$. Note that $E$ is the dual of a spinor bundle on $Q^8$ (see \cite[Definition 1.3 and Theorem 2.8]{ott2}) and hence $c_8(E)=1$ (see \cite[Remark 2.9]{ott2}). This means that a general hyperplane section of $S_4\subset\p^{15}$ contains exactly one $6$-dimensional quadric hypersurface, so in particular every $X\subset\p^{15-s}$ is not a strong QEL-manifold for $s\geq 1$, proving the claim and giving (iv). The cases where $b=2$ are ruled out by Theorem \ref{thm:(2,2)} and the last case was ruled out in Theorem \ref{thm:cartan}, so we focus on the case $b=3$, $\delta=0$, $s=2$ and $n=6$. In this case, we argue as in Theorem \ref{thm:spinor}. Consider $X_p\subset\mathcal L_z\subset\p^{13}$ where $\mathcal L_z\subset\p^{13}$ is a codimension-$2$ linear section of $S_4\subset\p^{15}$. For a general $p'\in\p^{13}$ the entry locus (relative to $\mathcal L_z$) is a $4$-dimensional quadric hypersurface $\Sigma_{p'}\subset\p_{p'}^5$ intersecting $X_p$ in a surface. Let $\p_{pp'}^6:=\langle p,\p_{p'}^5\rangle$. Then $\Phi:\p^{14}\da Z$ restricted to $\p_{pp'}^6\subset\p^{14}$ is a birational transformation onto a $6$-dimensional quadric hypersurface $\Phi(\p_{pp'}^6)\subset Z$, namely a fibre of the tangential projection of $Z$ onto an $8$-dimensional quadric hypersurface. Repeat the construction taking a general $p''\in\p^{13}$ such that $\Sigma_{p'}\cap\Sigma_{p''}$ is a $\p^2$ (note that this is always possible). Then we deduce that $\p_{pp'}^6\cap\p_{pp''}^6$ is a linear $\p^3$ in $\p^{14}$ and we conclude as in Theorem \ref{thm:spinor}.
\end{proof}

\begin{remark}\label{rem:hom}
Secant defective LQEL-manifolds are expected to be non-linearly normal secant defective QEL-manifolds (see Remark \ref{rem:LQEL}), and secant defective QEL-manifolds are expected to be (again by the lack of examples) linear sections of secant defective rational homogeneous varieties (cf. \cite[Remark 3.8]{i-r3}). If this conjecture is true then Corollary \ref{cor:homogeneous} gives a complete classification of special birational transformations of type $(2,b)$ onto prime Fano manifolds different from quadric hypersurfaces in the boundary case $\delta(Z)=2b+(2b-1)\delta$ of Lemma \ref{lem:bound}. Therefore, the following question arises:
\end{remark}

\begin{question}\label{q:QEL}
Is there any special birational transformation $\Phi:\p^r\da Z$ of type $(2,b)$ such that $\delta(Z)>2b+(2b-1)\delta$?
\end{question}

We believe that special birational transformations $\Phi:\p^r\da Z$ defined by quadric hypersurfaces are very rare objects, specially if $Z$ is neither a projective space nor a quadric hypersurface. This is due to a couple of facts. First, one needs to find a \emph{good candidate} base locus $X\subset\p^r$ giving a birational transformation (and hence satisfying Proposition \ref{prop:sec}). Secondly, and what is more important, one has to show that the resulting $Z$ is smooth. This is automatic if $Z=\p^r$, but in view of Proposition \ref{prop:X_p} it seems to be a very restrictive condition as soon as $Z\neq\p^r$. Concerning Question \ref{q:QEL}, the answer is negative for $b=1$, as in this case $X\subset\p^r$ is degenerate and $\mathcal L_z\subset\p^{r-1}$ is projectively equivalent to $X\subset\p^{r-1}$ by Corollary \ref{cor:L_z}(iv), and hence $\delta(Z)=2+\delta$ (cf. Remark \ref{rem:tangential}). On the other hand, for $b=2$ we present below some examples in which $X\subset\p^r$ is a good candidate and $\delta(Z)>2b+(2b-1)\delta$ but $Z$ turns out to be singular. In these examples $SX\subset\p^r$ is still a cubic hypersurface, so $Z$ is a normal variety of Picard number one and covered by lines with only $\q$-factorial and terminal singularities:

\begin{example}
(i) Let $X\subset\p^{(3n+s)/2+2}$ be a linear section of a Severi variety of codimension $2\leq s\leq\delta$. Then a basis of $I_X(2)$ gives a birational transformation of $\p^{(3n+s)/2+2}$ onto a singular complete intersection in $\p^{3(n+s)/2+2}$ of $s$ quadric hypersurfaces, as hyperplanes are mapped to quadric hypersurfaces in the Cremona transformation given by the Severi varieties. In this case, $\delta=(n-s)/2$ and $\delta(Z)=3+s+3\delta$.

(ii) Let $X\subset\p^{2n+2}$ be a hyperplane section of the Segre embedding $\p^1\times Q^n\subset\p^{2n+3}$, where $Q^n\subset\p^{n+1}$ is a quadric hypersurface of dimension $n\geq 2$. Then a basis of $I_X(2)$ gives a birational transformation of $\p^{2n+2}$ onto a singular complete intersection in $\p^{(n^2+3n+6)/2}$ of a quadric hypersurface and the cone over $\g(1,n+1)\subset\p^{(n^2+3n)/2}$ of vertex a plane. In this case, $\delta=0$ and $\delta(Z)=5$.
\end{example}

\section*{Acknowledgements}
We wish to thank Fyodor Zak and Francesco Russo for valuable discussions.

\bibliography{bibfile}
\bibliographystyle{amsplain}

\end{document}